\documentclass[12pt,reqno]{CAA}
\large\normalsize

\usepackage{hyperref}
\usepackage{graphicx}
\usepackage{amsmath}
\usepackage{booktabs} 
\usepackage{cite}

\newtheorem{theorem}{Theorem}[section]

\newtheorem{proposition}[theorem]{Proposition}

\newtheorem{example}[theorem]{Example}
\newtheorem{remark}[theorem]{Remark}

\numberwithin{equation}{section}
\newcommand{\CAAnewsection}[1]{\addtocounter{section}{1}\section*{#1}\setcounter{equation}{0}\setcounter{theorem}{0}}

\newcommand{\CAAnewsubsection}[1]{\subsection*{#1}\noindent}

\begin{document}

\noindent {{\it Communications in Applied Analysis} {\bf xx} (20xx), xxx--xxx}

\title[A model of a network community\hfill]{\mbox{}\\[1cm]A model of a network community \\
with a decentralized reputation-based peer control of information}

\author[\hfill A. Olifer]{Andrei Olifer}
\begin{abstract}
\vspace{-.3cm}

\begin{center}
School of Science and Technology, Georgia Gwinnett College\\ Lawrenceville, GA 30043, USA\\
{\em E-mail:} aolifer\symbol{64}ggc.edu\\[0.3cm]
\end{center}

\vspace{-.2cm}
\end{abstract}

\thanks{\hspace{-.5cm}\rm
	Received March 26, 2018
	\hfill 1083-2564 \$15.00 \copyright Dynamic Publishers, f.}
\maketitle
\thispagestyle{empty}
{\footnotesize
\noindent{\bf ABSTRACT.} 
This study was motivated by the problem of identifying fake documents on the Internet. To explore possible solutions to this problem we introduce a model of a network community in which members submit documents  with verifiable content. Documents are evaluated by three individuals, randomly selected from all the members. The members whose documents are evaluated as authentic gain reputation. Otherwise, they lose it. Members with higher reputation have a greater chance to be selected to serve as evaluators. Analytical and computational results suggest this evaluation mechanism is effective in a wide range of parameters including the presence of members who consistently submit fake content and do not evaluate documents of others objectively. The proposed mechanism can be used in a variety of situations including evaluating photo and video files with geolocation and time data and creating collections of digital evidence of events clear from fake materials.

}

\CAAnewsection{1. Introduction}
\label{intro}
The Internet is swamped with unverified and fake information. Distribution of such documents is widely used in politics and commerce \cite{CZ13,B14,AG17}. Known approaches to verification of information have limited efficiency. Fact checking sites such as FactCheck.org and Snopes.com have limited capabilities to counter-balance the influx of fake materials. The throughput of these sites is constrained by available resources, particularly by the  number of people working for them. In addition, the credibility of such cites is questioned since closed groups of people, involved in verification, may develop bias \cite{NP13}. An attractive alternative could be a decentralized, distributed mechanism of verification. Here we define and explore the viability of a model of a network community of members with such a mechanism.

The proposed verification process is inspired by academia peer-reviewing. The following principles of academic reviewing have been borrowed: blind evaluation by several independent reviewers and the selection of reviewers according to their reputation. 
The reviewers are independent in a sense that they do not communicate with each other or anybody else. The reviewing is blind -- reviewers do not know the author of the reviewed document. The selection according to reputation means that the community members with a greater reputation have a greater probability of being selected. However, the reviewing process modeled here is simpler than academic peer-reviewing in a few respects: 1) evaluation results are binary -- documents are evaluated as either authentic or fake; 2) the evaluation process is not hierarchical -- each of three reviewers votes pro or contra and the decision is determined by a simple majority; 3) evaluators do not deal with rebuttals and resubmissions. 

Modeling of reputations is the core of the proposed model. There are many computational models of reputation and trust that vary in how individuals interact and use the results of interactions; see for example review  \cite{PS13}. Particular attention has been placed on peer-to-peer interactions in the context of on-line commerce as in eBay and Amazon \cite{XL03,CB13,KSH03,MH02}. The model considered in this study is different from these works. Most importantly, the model describes the time evolution of {\it{proportions}} of individuals with certain levels of reputation. It does not take into account the reputations of specific individuals. Similar mechanisms of reputation changes have been considered earlier \cite{KSH03}, but as far as we know, not in population models. The population model approach allows us to explore the behavior of large communities. In traditional agent-based simulations the size of model communities is limited.

The main question of the study was whether a community with the specified rules could function in such a way that the authenticity of most of the documents submitted by its members were evaluated correctly. Here, we started the exploration of the community's model assuming simple, if not the simplest, choices for the parameters that could be set by community developers. We show that for those choices, under some assumptions regarding members' behavior, the answer is affirmative.

The model and approach are introduced in Section 2. In that section, a simple case of a community without cliques and members with only three possible values of reputation is considered. Section 3 describes a generalization of the model to a case with an arbitrary number of reputation values. Results of computer simulations of the model with eleven possible values of reputation are presented in Section 4. Subsequent Sections 5 and 6 explore modeled communities with one and then two cliques respectively; in the latter case, the cliques are antagonistic to each other. The paper ends with a discussion (Section 7) and conclusions (Section 8).

Computer simulations were performed in Matlab (MathWorks, Natick, MA) using ode45 function with the absolute and relative tolerances equal to $10^{-7}$ on a PC with a 2.3 GHz processor and 16 Gb memory.     
 
\CAAnewsection{2. Communities with three levels of reputation}
\label{Communities with three levels of reputation}

\CAAnewsubsection{1. Notations and equations}
In this section, for simplicity of exposition, the reputation of the community members has only three values, $r_0$, $r_1$, $r_2$, where $0 \leq r_0<r_1<r_2 \leq 1$. 
The values of $r_k$ are sometimes expressed as percentages. For example, $r_k=1$ corresponds to 100\% reputation.
The proportion of the community members with the reputation $r_k$ is denoted by $R_k$, $0\leq R_k \leq 1$, $k=0,1,2$; $R_0+R_1+R_2=1$. Equivalently, $R_k$ is the probability of randomly selecting a member with the reputation $r_k$. Together, $R_k$ define a probability distribution. It is assumed that new members start with one and the same value of reputation. Appearance of new members and loss of existing members of the community can be modeled by changes in the corresponding values of $R_k$. For example, if there is an influx of new members with the reputation $r_1$ then the distribution of $R_k$ should change to reflect the increase of $R_1$. However, in this study we assume that the number of the community members is fixed.

After the document, submitted by a member, is evaluated, the member's reputation changes. It increases or remains maximal (equal to $r_2$) if the document is evaluated as authentic. Otherwise the reputation decreases or remains minimal (equal to $r_0$). All the members submit documents with the same rate. These assumptions lead to the equations

\begin{equation}
\label{case3}
\begin{split}
\frac{dR_{0}}{dt} & = \underbrace{-R_{0}e_0}_{\substack{increase\\reputation}}+\underbrace{R_{1}(1-e_1)}_{\substack{decrease\\reputation}}\\
\frac{dR_{1}}{dt} & = \underbrace{R_{0}e_0 -R_{1}e_1}_{\substack{increase\\reputation}} +  \underbrace{R_{2}(1-e_2) -R_{1}(1-e_1)}_{\substack{decrease\\reputation}}\\
\frac{dR_{2}}{dt} & = \underbrace{R_{1}e_1}_{\substack{increase\\reputation}}-\underbrace{R_{2}(1-e_2)}_{\substack{decrease\\reputation}},
\end{split}
\end{equation} 

\noindent
where $t$ is time, and $e_k$, $0\leq e_k \leq 1$, is the probability that a document, submitted by a member with the reputation $r_k$, is evaluated as authentic during a unit of time.

The following proposition shows that (\ref{case3}) indeed describes an evolution of a probability distribution $\{R_0, R_1, R_2\}$.

\begin{proposition}
	\label{propo_case3}
	Let the initial values $R_k(0)$ of $R_k$ be such that $R_0(0)+R_1(0)+R_2(0)=1$ and $0\leq R_k(0)\leq 1$. Then a) $R_0(t)+R_1(t)+R_2(t)\equiv 1$, and b) $0\leq R_k(t)\leq 1$  for all $t$.  
\end{proposition}

\begin{proof} Summing up the equations yields $d(R_{0}+R_{1}+R_{2})/dt\equiv 0$. Hence at any moment of time, as at $t=0$, the sum of $R_k$ is equal to one  which proves a). Part b) follows from the fact that according to (\ref{case3}) once $R_k$ becomes equal to $0$ it can no longer decrease and once it becomes equal to $1$ it can no longer increase. Indeed, let for example, $R_0(t)=0$ for some $t$, while $0<R_1(t)<1$ and $0<R_2(t)<1$. Then the right hand side of the first equation in (\ref{case3}) is non-negative, or $dR_0(t)/dt\geq 0$. On the other hand, if $R_0(t)=1$ for some $t$ then it implies $R_1(t)=0$ and the right hand side of the equation for $dR_0/dt$ is non-positive, or $dR_0(t)/dt\leq 0$. The cases of $R_1$ and $R_2$ are similar. 
\end{proof}

To specify system (\ref{case3}) completely we need to define  the probabilities $e_k$. This takes several steps. The probability $s_k$ of selecting a member with a reputation $r_k$ to serve as a reviewer is set to be proportional to the member's reputation $r_k$:
\begin{equation}
\label{s_k}
s_k=\frac{r_kR_k}{r_{0}R_{0}+r_{1}R_{1}+r_{2}R_{2}},\qquad k=0,1,2.
\end{equation}

For $r_0=0$ the formula gives $s_0=0$ meaning that the members with zero reputation are never selected to be evaluators. Function $s_0(R_0,R_1,R_2)$ has a discontinuity at $R_0=1$, $R_1=0$, $R_2=0$.  We remove it by setting $s_0=0$ at this point.

Let $p_{c,ind}$ be the probability that a randomly selected evaluator correctly evaluates documents, and let a member with a reputation $r_k$ correctly evaluate documents with a probability $c_k$. Since a selected reviewer can have one of the three reputations,
\begin{equation}
\label{pc_knd}
\begin{split} 
p_{c,ind}& =s_0c_0+s_1c_1+s_2c_2.\\
\end{split}
\end{equation}

In this and the other models considered in this study whether the document is authentic or fake is independently decided by three reviewers. As with the choice of the values of $r_k$ we wanted to start the exploration of the model with a simple case. Three reviewers is the smallest number of reviewers who could make a decision by the majority without a tie. 

The probability that the document is correctly evaluated by the majority of reviewers is 
\begin{equation}
\label{p_c}
p_c  = p^3_{c,ind} + 3p^2_{c,ind}(1-p_{c,ind}).
\end{equation} 

The first term is the probability that all three evaluators correctly evaluate the document. The second term is the probability that one of the evaluators makes a mistake.

 The probabilities $e_k$, that the documents submitted by the members with reputation $r_k$ are evaluated as authentic, are sums of two probabilities. One is a product of the probability $a_k$ that the document submitted by a member with reputation $r_k$ is  authentic and the probability $p_c$ that it is correctly evaluated as authentic. The other probability is the product of the probability $1-a_k$ that the submitted document is fake and the probability $1-p_c$ that it is incorrectly evaluated as authentic:

\begin{equation}
\label{e_k}
e_k = a_k p_c + (1-a_k)(1-p_c), \qquad k=0,1,2. 
\end{equation}

Let the probabilities $a_k$ in (\ref{e_k}) and $c_k$ in (\ref{pc_knd}) be determined by functions of the members' reputation $a=a(r)$ and $c=c(r)$ and that $a_k=a(r_k)$ and $c_k=c(r_k)$. We assume that the probabilities of submitting authentic documents, $a(r)$, and correctly evaluating documents, $c(r)$, increase with the increase of reputation. We also posit that the members with zero reputation never submit authentic documents, $a(0)=0$, and never evaluate the documents of others correctly, $c(0)=0$, while the members with 100\% reputation always submit authentic documents, $a(1)=1$, and always evaluate the documents of others correctly, $c(1)=1$. Let $a(r)$ and $c(r)$ be at most quadratic functions of $r$. It can be directly verified that to satisfy the assumptions above they have to have the following form: 

\begin{equation}
\label{a_c_k}
\begin{split}
a(r) & =r(1+\alpha(1-r)), \qquad -1\leq \alpha \leq 1,\\
c(r) & =r(1+\sigma(1-r)), \qquad -1\leq \sigma \leq 1.\\
\end{split}
\end{equation}
\noindent 
When parameters $\alpha$ and $\sigma$ are positive, the functions $a(r)$ and $c(r)$ are concave (Fig.\ref{fig_auth}). In this case the probabilities of submitting authentic documents and correct document evaluation exceed the reputations of the members. When $\alpha$ and $\sigma$ are negative the functions $a(r)$ and $c(r)$ are convex and the probabilities of submitting authentic documents and correct document evaluation are less than the reputations of the members. (In the future studies it would be interesting to explore a model with more general assumptions $a(0), c(0)\geq 0$ and $a(1), c(1)\leq 1$). 

\begin{figure}[!ht]
	\centering
	\includegraphics[scale=0.6]{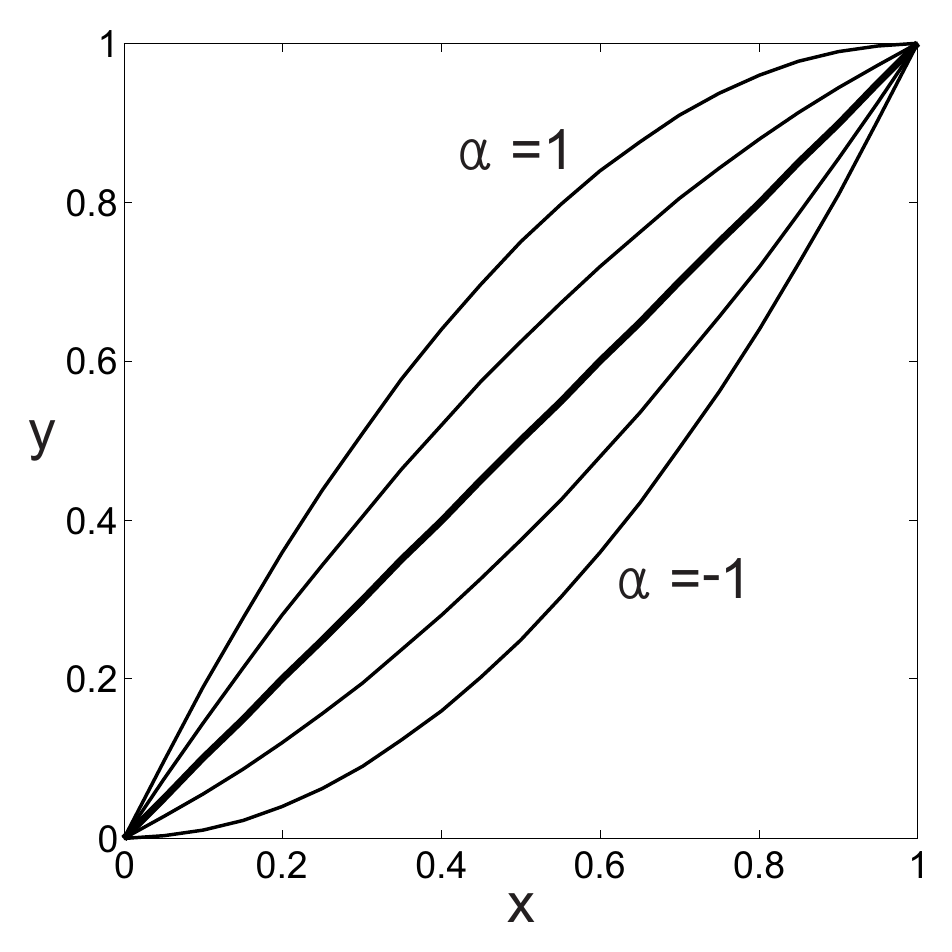} 
	\caption{Function $f(x)=x(1+\alpha(1-x))$ used in (\ref{a_c_k}). The line $y=x$ corresponds to $\alpha=0$. The two other unlabeled curves correspond to the $\alpha=0.5$ and $\alpha=-0.5$.}
	\label{fig_auth}
\end{figure}

System (\ref{case3}) is completely defined now. It is a non-linear system with five parameters: $r_0$, $r_1$, $r_2$, $\alpha$ and $\sigma$. Unless mentioned otherwise $r_0=0$, $r_1=0.5$, $r_2=1$.
We chose these values for two reasons. They simplify analysis, and for these values of $r_k$ the modeled system converges to a state at which all documents are evaluated correctly; see Proposition \ref{propo_lnh_equil} below.

\CAAnewsubsection{2. Equilibria}
System (\ref{case3}) is essentially two-dimensional since one of the variables $R_k$ can be excluded using the equality $R_0+R_1+R_2\equiv 1$. With the excluded $R_1$ system (\ref{case3}) turns into

\begin{equation}
\label{case3_2}
\begin{split}
\frac{dR_{0}}{dt} & = -R_{0}e_0+ (1-R_{0}-R_{2})(1-e_1)\\
\frac{dR_{2}}{dt} & = (1-R_{0}-R_{2})e_1-R_{2}(1-e_2).
\end{split}
\end{equation}

The Poincar\'e-–Bendixson theorem \cite{H02} implies that the long-term behavior of system (\ref{case3_2}) as a two-dimensional system, could be either an equilibrium, a limit cycle or a homo-/heteroclinic orbit.

The vector field of system (\ref{case3_2}) for different values of $\alpha$ and $\sigma$ indicates that the system has only equilibria that are located at the line $R_2=1-R_0$ (Fig.\ref{fig_vector_fields}). 

\begin{proposition} 
	\label{propo_lnh_equil}
	The points of the line $R_2=1-R_0$, $R_0\in [0,1]$, are equilibria of system (\ref{case3_2}). At these points all documents are evaluated correctly, $p_c=1$.
\end{proposition}

\begin{proof} For the proof we need to verify that $-R_{0}e_0+ (1-R_{0}-R_{2})(1-e_1)=0$ and $(1-R_{0}-R_{2})e_1-R_{2}(1-e_2)=0$. The condition $R_2=1-R_0$ simplifies the equations to $R_0e_0=0$ and $R_2(1-e_2)=0$. The condition also implies $R_1=0$. It is directly verifiable that (\ref{s_k})-(\ref{a_c_k}) yield $p_c=1$, $a_0=0$, $a_2=1$, and then $e_0=0$ and $e_2=1$, which proves the proposition.
\end{proof}

In fact the proof holds for all the functions $a(r)$ and $c(r)$ such that $a(0)=0$ and $a(1)=c(1)=1$. The condition $a(0)=0$ implies that the members with zero reputation submit only fake documents. The condition $a(1)=c(1)=1$ implies that the members with 100\% reputation submit only authentic documents and make no mistakes in evaluation. 

The reasons why the states ($R_0, 1-R_0$) are equilibria are quite intuitive. At equilibrium, the evaluators are chosen only from the members with 100\% reputation who make no mistakes. Hence the members with zero reputation, who submit only fake documents, have no chance to increase their reputation, while the members with 100\% reputation, who submit only authentic documents, never loose it. It is not the case when there exist members whose reputation is more than zero and less than one. Such members can be selected to evaluate and can make mistakes in evaluation. Because of such mistakes, members with zero reputation can increase it and members with 100\% reputation can loose it.

Linear stability analysis does not allow to determine the stability of the equilibria $(R_0,1-R_0)$. When $R_0<1$, the matrix of the system (\ref{case3_2}), linearized at $(R_0,1-R_0)$, has the eigenvalues $-1$  and $0$, the latter meaning a critical case \cite{A12}. Computer simulations suggest that the equilibria $(R_0,1-R_0)$, $R_0\in[0,1]$, are stable in Lyapunov's sense, meaning that if the system starts evolving from near an equilibrium it will remain near the equilibrium forever \cite{A12} (cf. Fig.\ref{fig_vector_fields}).

\begin{figure}[!h]
	\centering
	\includegraphics[scale=0.7]{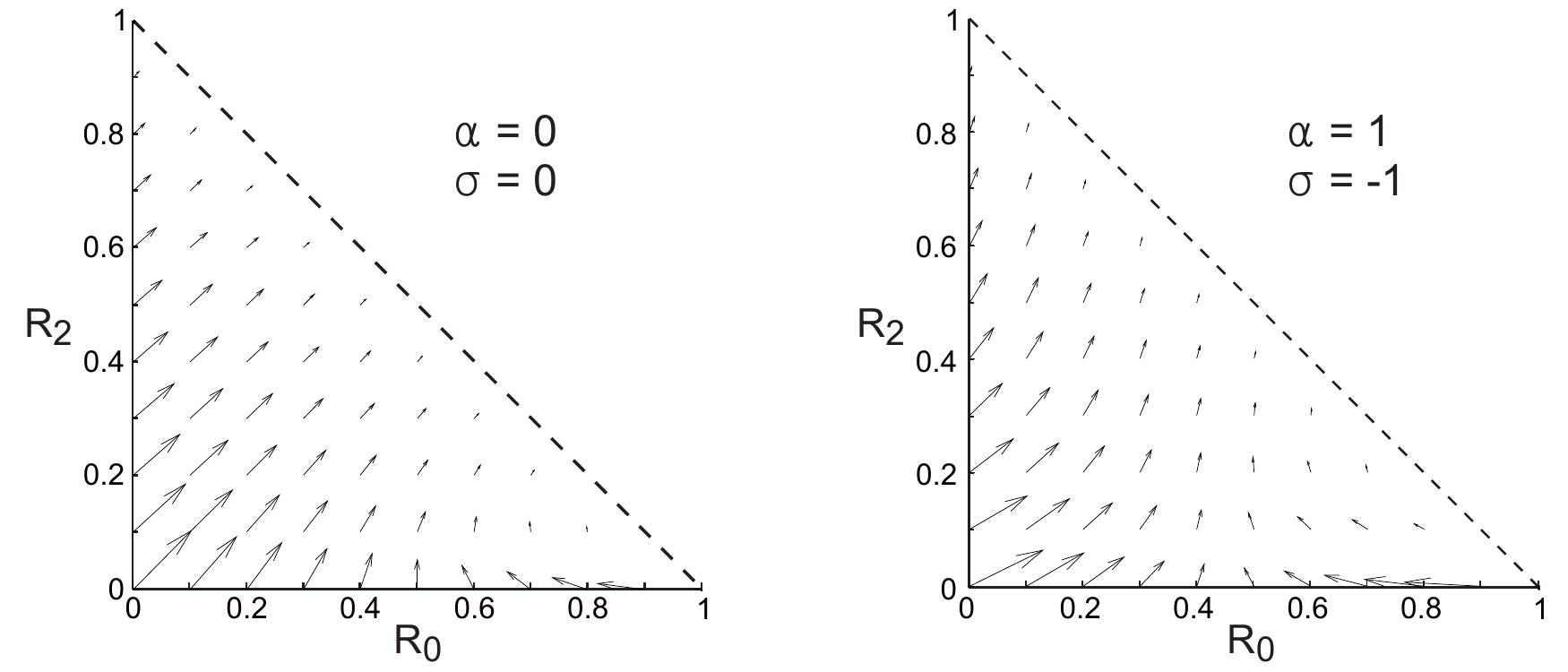} 
	\caption{The vector field of model (\ref{case3_2}). $\alpha=\sigma=0$ (left), $\alpha=1$, $\sigma=-1$ (right).}
	\label{fig_vector_fields}
\end{figure}

\CAAnewsection{3. Communities with multiple levels of reputation}
\label{Communities with multiple levels of reputation}
System (\ref{case3}) can be directly generalized to the case with more than three levels of reputation. Let $r_k=\Delta k$, $\Delta =1/L$, $k=0,\cdots, L$. (Equally spaced values of $r_k$ are defined by just one parameter, $L$. This is important at this initial stage of the study since the model has a large number of other parameters. It is possible that for other choices of $r_k$ the model performs better. This is a question for future research.) The equations for $R_k$, $0\leq R_k\leq 1$, $\sum_{k=0}^LR_k=1$, take the form

\begin{equation}
\label{caseL}
\begin{split}
\frac{dR_{0}}{dt} & = -R_{0}e_0+R_{1}(1-e_1)\\
& \vdots \\
\frac{dR_{k}}{dt} & = R_{k-1}e_{k-1}-R_{k}+R_{k+1}(1-e_{k+1}), \qquad k =1,\cdots,L-1,\\
& \vdots \\
\frac{dR_{L}}{dt} & = R_{L-1}e_{L-1}-R_{L}(1-e_{L}).
\end{split}
\end{equation} 

Definitions (\ref{e_k})-(\ref{a_c_k}) are generalized in a straightforward way. As in the case of system (\ref{case3_2}), system (\ref{caseL}) has two parameters: $\alpha$ and $\sigma$.

System (\ref{caseL}) has a similar set of equilibria as system \ref{case3_2} for the community with three levels of reputation.

\begin{proposition} 
	\label{propo_caseL_equil}
	The points of the line $(R_0,0,\cdots,0,1-R_0)$, $R_0\in [0,1]$, are equilibria of system (\ref{caseL}). At these points all documents are evaluated correctly, $p_c=1$.
\end{proposition}

\begin{proof} For the proof we need to verify that on the line $(R_0,0,\cdots,0,1-R_0)$ all the right-hand sides of (\ref{caseL}) are equal to zero. This is trivially true for all the equations but the first one, 
	$R_{0}e_0=0$, and the last one, $R_L(1-e_L)=0$. For these two equations the reasoning is the same as in the proof of Proposition \ref{propo_lnh_equil}.
\end{proof}

The numerical results presented in the next section suggest that as in the case of system (\ref{case3_2}) these equilibria are stable in Lyapunov's sense for all $\alpha, \sigma \in [-1,1]$. However they are not globally stable at least for some values of $\alpha$ and $\sigma$.  

\CAAnewsection{4. Computer simulations}
\label{Computer simulations}
In what follows $L=10$ and $r_k=0.1k$, $k=0,\cdots, 10$. Since the dynamical system is ten--dimensional (only ten independent $R_k$s) the Poincar\'e-–Bendixson theorem is not applicable. Figure \ref{fig_reps10_T100_Rk.pdf} shows the dynamics of the system, typical for the majority of the initial values and values of $\alpha$ and $\sigma$. Two `forces' pull $R_k$s in the opposite directions -- the members either gain reputation until it reaches the maximal value of 100\%, or loose reputation until it reaches its minimal value of 0\%. The most common initial condition considered is supposed to model the `birth' of the community, when it consists of new members only, and every new member has the same initial reputation that we set to be 60\%.

\begin{figure}[!th]
	\centering
	\includegraphics[scale=0.7]{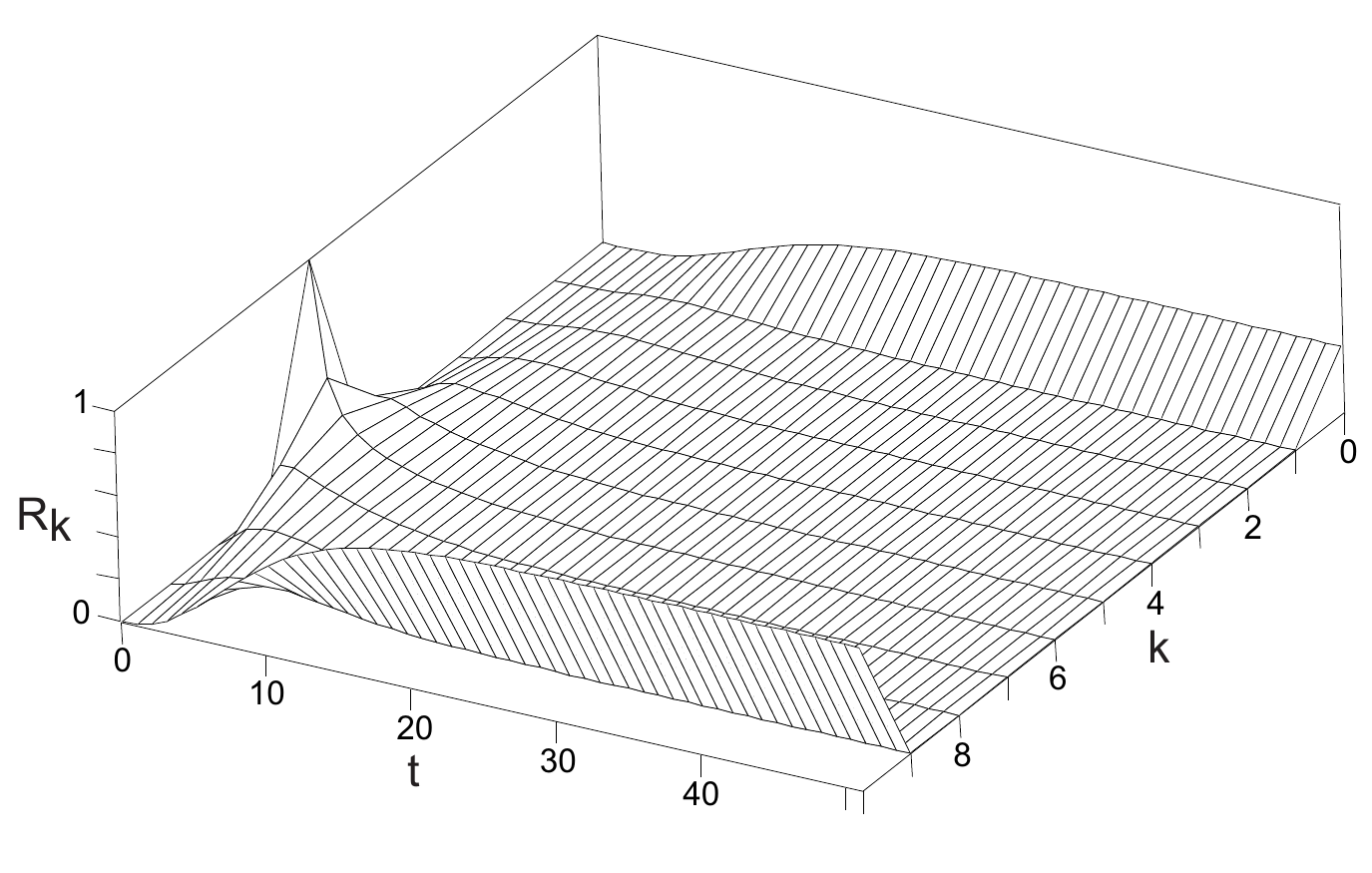} 
	\caption{Evolution of the model with eleven levels of reputation to a bimodal distribution. $R_k$ are the proportions of the members with reputations $r_k=0.1k$, $k=0,\cdots,10$. Initially, at $t=0$, $R_6=1$, while all the other $R_k=0$. Over time, all the members attain either zero ($k=0$) or 100\% ($k=10$) reputation, and the probability of correct document evaluation $p_c$ approaches one. $\alpha=\sigma=0$.}
	\label{fig_reps10_T100_Rk.pdf}
\end{figure}

The dynamics of system (\ref{caseL}) is more complex compared to system (\ref{case3}). For $\alpha$ and $\sigma$ near $\alpha=1$, $\sigma=-1$ the system shows bi-stability -- the equilibrium state depends on initial values of $R_k$. The existence of an equilibrium different from those described in Proposition \ref{propo_caseL_equil} is shown in Figure \ref{fig_reps10_T100_small_Rk.pdf}. The initial conditions are the same as in the case of Fig.\ref{fig_reps10_T100_Rk.pdf}  but $\alpha=1$, $\sigma=-1$. As the figure shows, the `force' that makes members of the community lose their reputation dominates and the equilibrium is attained with the majority of the members having low reputations; the most numerous category of the members have the reputation equal to 30\%. At that equilibrium state $p_c=0.07$ meaning evaluators are almost always wrong. However, with the initial values $R_k=0$, $k\neq 7$, and $R_7=1$, the system converges to a bimodal distribution $(0.0425,0,\cdots,0, 0.9575)$ with $p_c=1$ (not shown). The probability of correct document evaluation $p_c$ at the equilibrium states of system (\ref{caseL}) and various values of $\alpha$ and $\sigma$ is shown in Fig.\ref{fig_reps10_T100_Pc_v2.pdf}. For all the values of $\alpha$ and $\sigma$ initial values are the same as in Figures \ref{fig_reps10_T100_Rk.pdf} and \ref{fig_reps10_T100_small_Rk.pdf}.

Computations show that whether and how the system exhibits bi-stability depends on the dimension of the system, and parameters  $\alpha$ and $\sigma$. In these computations, all the members had initial reputation equal to 0.6. In the case of the 5--dimensional system with the reputation values $r=(0,0.2,\cdots,1)$ the system behaves as in Fig.\ref{fig_reps10_T100_Rk.pdf} and converges to an equilibrium state with $p_c$=1. In the case of the 20--dimensional system with the reputation values $r=(0,0.05,\cdots,1)$ the system behaves as in Fig.\ref{fig_reps10_T100_small_Rk.pdf} and converges to an equilibrium state with $p_c=0.04$. Making all initial reputations in the 20--dimensional system equal to 0.65 does not change the behavior quantitatively. However, the equilibrium state changes drastically when all initial reputations are set to 0.7. Then the system behaves as in Fig.\ref{fig_reps10_T100_Rk.pdf} with an equilibrium state at which $p_c=1$. 

\begin{figure}[!th]
	\centering
	\includegraphics[scale=0.75]{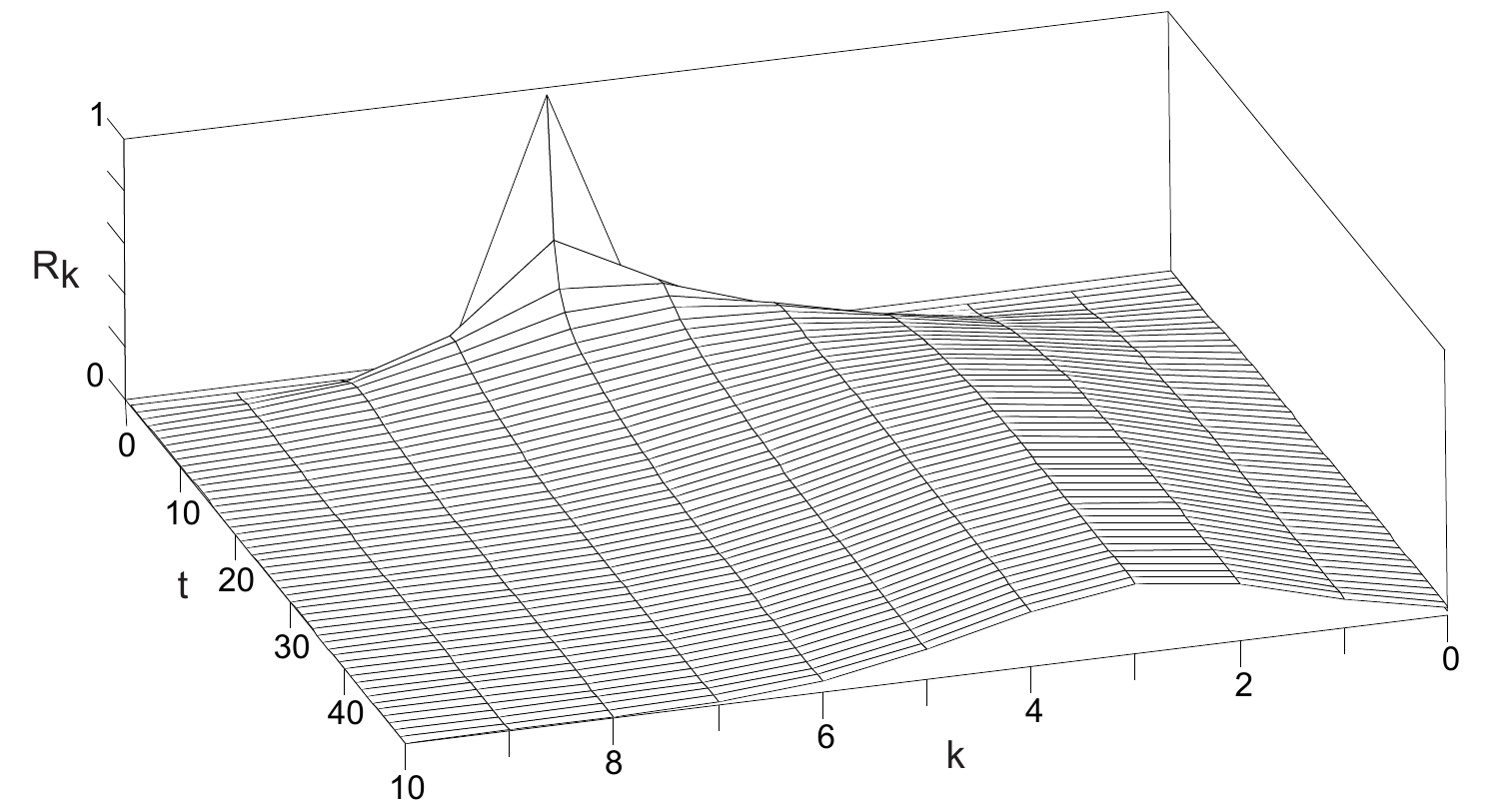} 
	\caption{Evolution of the model with eleven levels of reputation to a unimodal distribution. $R_k$ are the proportions of the members with reputations $r_k=0.1k$, $k=0,\cdots,10$. At $t=0$, $R_6=1$ while all the other $R_k=0$. At the equilibrium, most of the members have low reputations. The members in the most numerous category have reputation 30\% ($k=3$). At this state the probability of correct document evaluation $p_c=0.07$. $\alpha=1$, $\sigma=-1$.}
	\label{fig_reps10_T100_small_Rk.pdf}
\end{figure}

\begin{figure}[!th]
	\centering
	\includegraphics[scale=0.75]{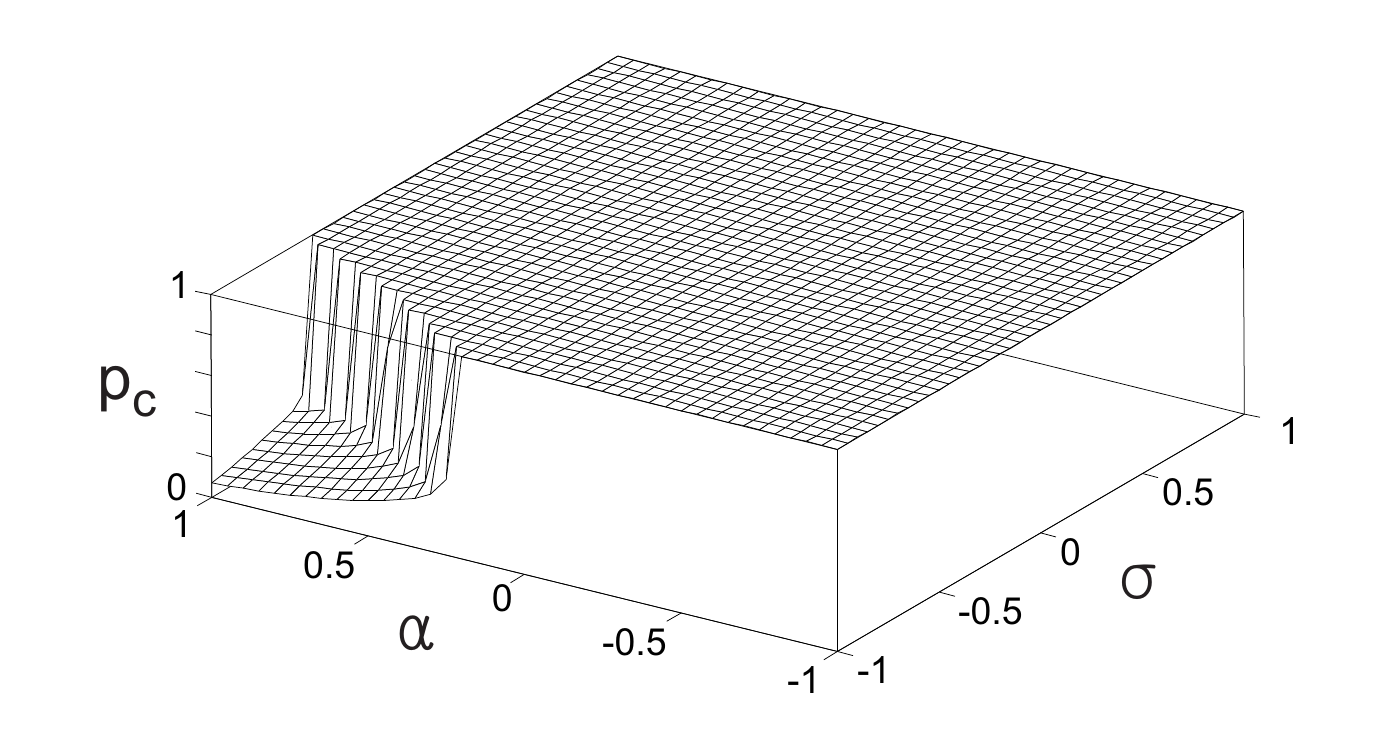} 
	\caption{Probability of correct document evaluation $p_c$ in the model of a community without cliques is equal to one for most values of parameters $\alpha$ and $\sigma$. At $t=0$, $R_6=1$ while all the other $R_k=0$.}
	\label{fig_reps10_T100_Pc_v2.pdf}
\end{figure}

In summary, simulations suggest that in the absence of cliques the behavior of the modeled community evolves to one of two stable steady states. In one state, all the members have either 100\% reputation or zero reputation; see example in Fig.\ref{fig_reps10_T100_Rk.pdf}. In this state, all the documents are evaluated correctly for a majority of the values of $\alpha$ and $\sigma$ that specify functions $a(r)$ and $c(r)$ (\ref{a_c_k}). In the other state, the majority of community members have low reputations; see example in Fig. \ref{fig_reps10_T100_small_Rk.pdf}. In this state, the documents are rarely evaluated correctly. The system evolves to one state or another depending on initial conditions which in turn depend on the dimension of the system. Detailed analysis of this bi-stability was beyond the scope of the present study.

\CAAnewsection{5. Community with a clique}
\label{Community with a clique}
In real life there are groups of agents with agendas. For those groups their agendas are the only thing that matters. Groups of this kind are called cliques here. Consider a system with one clique. Let $f_{cl}$ be the fraction of clique members in the community. The remaining $1-f_{cl}$ fraction are `regular' members.
We extend the system with regular members, considered in the previous section, by adding $L+1$ phase variables $Q_k$, $k=0,\cdots,L$, $0\leq Q_k \leq 1$, for the proportions of the clique members with reputations $r_k$, $Q_0+Q_1+\cdots+Q_L=f_{cl}$.

Table \ref{1cl_subm} defines the probabilities of submitting documents of various types by regular and  clique members. Parameter $p_{\lambda}$ stands for the probability that a document submitted by a regular member supports the clique's views. For simplicity, the probability that a document submitted by a regular member contradicts the clique's views is set to $p_{\lambda}$ as well. The probability that a document has nothing to do with the clique's agenda is $1-2p_{\lambda}$. Clique members submit only the documents that support the clique's view. Parameter $\gamma$, $0\leq \gamma \leq 1$, determines to what extent clique members are able to submit authentic documents in accord with their reputation; for small values of $\gamma$, even highly reputable clique members submit mostly fake documents.  Probabilities $a_k$ of submitting authentic documents by the members with reputation $r_k$ are defined by (\ref{a_c_k}).

\begin{table}
	\vspace{\baselineskip}
	\caption{Probabilities of {\emph{submitting}} various types of documents by regular and clique members with reputations $r_k$, $k=0,\cdots, L$. }  
	\vspace{\baselineskip}
	\begin{center}
		\begin{tabular}{l *{2}{r}}
			\toprule
			Document type & Regular members & Clique members\\
			\bottomrule
			Generic, authentic & $(1-2p_{\lambda}) a_k$ & $0$ \\
			\midrule
			Generic, fake & $(1-2p_{\lambda})(1-a_k)$ &  $0$\\
			\midrule
			Clique, authentic & $p_{\lambda}a_k$ & $\gamma a_k$ \\
			\midrule
			Clique, fake & $p_{\lambda}(1-a_k)$ & $1-\gamma a_k$ \\
			\midrule
			Anti-clique, authentic & $p_{\lambda}a_k$ & $0$ \\
			\midrule
			Anti-clique, fake & $p_{\lambda}(1-a_k)$ & $0$ \\
			\bottomrule
		\end{tabular}
	\end{center}
	\label{1cl_subm}
\end{table}
\vspace{\baselineskip}

Table \ref{1cl_eval} defines the probabilities of correct evaluation of particular documents by regular and clique members. The probabilities $c_k$ of correct document evaluation by the members with reputations $r_k$ are defined by (\ref{a_c_k}). In contrast to regular members, who evaluate all the documents in the same way, clique members evaluate documents differently depending whether the documents support the clique's views or not. We assume they evaluate generic documents by flipping a coin because they do not care about such documents. All the documents that support the clique's views are evaluated as authentic. All the documents that contradict the clique's views are evaluated as fake.   

\begin{table}
	\vspace{\baselineskip}
	\caption{Probabilities of correct document {\emph{evaluation}} by regular and clique members with reputations $r_k$, $k=0,\cdots,L$. }  
	\vspace{\baselineskip}
	\begin{center}
		\begin{tabular}{l *{2}{r}}
			\toprule
			Document type & Regular members & Clique members\\
			\bottomrule
			Generic, authentic & $c_k$ & $0.5$ \\
			\midrule
			Generic, fake & $c_k$ &  $0.5$\\
			\midrule
			Clique, authentic & $c_k$ & $1$ \\
			\midrule
			Clique, fake & $c_k$& 0 \\
			\midrule
			Anti-clique, authentic & $c_k$ & $0$ \\
			\midrule
			Anti-clique, fake & $c_k$ & 1 \\
			\bottomrule
		\end{tabular}
	\end{center}
	\label{1cl_eval}
\end{table}

The differential equations for $Q_k$ are similar for those for $R_k$:

\begin{equation}
\label{caseL_1cl}
\begin{split}
\frac{dR_{0}}{dt} & = -R_{0}e_0^{(reg)}+R_{1}(1-e_1^{(reg)})\\
& \vdots \\
\frac{dR_{k}}{dt} & = R_{k-1}e_{k-1}^{(reg)}-R_{k}+R_{k+1}(1-e_{k+1}^{(reg)}), \qquad k =1,\cdots,L-1,\\
& \vdots \\
\frac{dR_{L}}{dt} & = R_{L-1}e_{L-1}^{(reg)}-R_{L}(1-e_{L}^{(reg)})\\
\frac{dQ_{0}}{dt} & = -Q_{0}e_0^{(cl)}+Q_{1}(1-e_1^{(cl)})\\
& \vdots \\
\frac{dQ_{k}}{dt} & = Q_{k-1}e_{k-1}^{(cl)}-Q_{k}+Q_{k+1}(1-e_{k+1}^{(cl)}), \qquad k =1,\cdots,L-1,\\
& \vdots \\
\frac{dQ_{L}}{dt} & = Q_{L-1}e_{L-1}^{(cl)}-Q_{L}(1-e_{L}^{(cl)}).
\end{split}
\end{equation} 

In (\ref{caseL_1cl}), $e_k^{(reg)}$, $0\leq k \leq L$, stand for the probability that the documents submitted by  the regular members with the reputation $r_k$ are evaluated as authentic. Similarly, $e_k^{(cl)}$, $0\leq k \leq L$, stand for the probabilities that the documents, submitted by clique members with the reputation $r_k$, are evaluated as authentic. As in (\ref{case3}) and (\ref{caseL}) the summation of all the equations in  (\ref{caseL_1cl}) shows that the sum of all $R_k$ does not change over time and is always equal to $1-f_{cl}$, and the sum of all $Q_k$ does not change over time, and is always equal to $f_{cl}$.

In the model of a community without cliques, the probability $e_k$ of evaluating documents as authentic is equal to the sum of the probabilities that authentic and fake documents submitted by the members with reputation $r_k$ are evaluated as authentic. In the model with a clique we consider authentic and fake documents of three possible types: those that have nothing to do with the clique's agenda (`Generic' or `g'), those that support the clique's views (`Clique' or `q'), and those that contradict the clique's views (`Anti-clique' or `\=q'). Each of these three types of documents can be authentic (`A') or fake (`F'). Table \ref{1cl_subm} defines the probability of evaluating a document of each category as authentic. Following the table we obtain for regular agents 

\begin{equation}
\label{e_k_n}
\begin{split}
e_k^{(reg)} & = (1-2p_{\lambda})a_k p_c(g,A)+(1-2p_{\lambda})(1-a_k)p_m(g,F) \\
& + p_{\lambda}a_k p_c(q,A) +p_{\lambda}(1-a_k)p_m (q,F) \\
& + p_{\lambda}a_k p_c(\bar{q},A)+ p_{\lambda}(1-a_k) p_m(\bar{q},F),  \qquad k=0,\cdots,L.
\end{split}
\end{equation}

Every term is a product of three probabilities. In the first term, these are the probabilities that the document is generic/authentic/correctly evaluated as authentic: $1-2p_{\lambda}$/$a_k$/$p_c(g,A)$.  In the second term, these are the probabilities that the document is generic/fake/mistakenly evaluated as authentic: $1-2p_{\lambda}$/$1-a_k$/$p_m(g,F)$. In the third term, these are the probabilities that the document supports the clique's views/authentic/correctly evaluated as authentic: $p_{\lambda}$/$a_k$/$p_c(q,A)$. In the fourth term, these are the probabilities that the document supports the clique's views/fake/mistakenly evaluated as authentic: $p_{\lambda}$/$1-a_k$/$p_m(q,F)$. In the fifth term, these are the probabilities that the document contradicts the clique's views/authentic/correctly evaluated as authentic: $p_{\lambda}$/$a_k$/$p_c(\bar{q},T)$. Finally, in the sixth term, these are the probabilities that the document  contradicts the clique's views/fake/mistakenly evaluated as authentic: $p_{\lambda}$/$1-a_k$/$p_m(\bar{q},F)$.

Taking into account that clique members submit only documents related to clique agenda (cf. Table \ref{1cl_subm})
\begin{equation}
\label{e_k_c}
e_k^{(cl)} = \gamma a_k p_c(q,A) + (1-\gamma a_k) p_m(q,F),  \qquad k=0,\cdots,L.
\end{equation}

The probability of correct evaluation, $p_c$, and the probability of false evaluation, $p_m$, in (\ref{e_k_n}) and (\ref{e_k_c})  depend on the corresponding probabilities for individual evaluators (cf. (\ref{p_c})). Namely,

\begin{equation}
\begin{aligned}
p_c(g,A) & = p_{c,ind}^2(g,A)(3-2p_{c,ind}(g,A)), &p_c(g,F) & = p_{c,ind}^2(g,F)(3-2p_{c,ind}(g,F)),\\ 
p_c(q,A) & = p_{c,ind}^2(q,A)(3-2p_{c,ind}(q,A)), &p_c(q,F) & = p_{c,ind}^2(q,F)(3-2p_{c,ind}(q,F)),\\
p_c(\bar{q},A) & = p_{c,ind}^2(\bar{q},A)(3-2p_{c,ind}(\bar{q},A)), &p_c(\bar{q},F) &= p_{c,ind}^2(\bar{q},F)(3-2p_{c,ind}(\bar{q},F)), 
\end{aligned}
\end{equation}

\noindent
where, following (\ref{pc_knd}) and Table \ref{1cl_eval},

\begin{equation}
\label{pcind_cl} 
\begin{aligned}
p_{c,ind}(g,A) & = \sum_{k=0}^Ls_k^{(reg)} c_k+0.5\sum_{k=0}^Ls_k^{(cl)},     &p_{m,ind}(g,F) &=\sum_{k=0}^Ls_k^{(reg)} (1-c_k)+0.5\sum_{k=0}^Ls_k^{(cl)},\\
p_{c,ind}(q,A) & = \sum_{k=0}^Ls_k^{(reg)} c_k+ \sum_{k=0}^Ls_k^{(cl)},     &p_{m,ind}(q,F) &= \sum_{k=0}^L s_k^{(reg)} (1-c_k)+\sum_{k=0}^Ls_k^{(cl)},\\
p_{c,ind}(\bar{q},A) & = \sum_{k=0}^Ls_k^{(reg)} c_k,  &p_{m,ind}(\bar{q},F) &= \sum_{k=0}^Ls_k^{(reg)}(1-c_k).
\end{aligned}
\end{equation}

\noindent
In (\ref{pcind_cl}), the probabilities of correct evaluation $c_k$ are defined as in (\ref{a_c_k}). 

The probability  of selecting a regular member with a reputation $r_k$ is  

\begin{equation}
\label{s_n}
s_k^{(reg)}= \frac{r_kR_k}{\sum_{i=0}^Lr_i (R_i +Q_i)}, 
\end{equation}

\noindent
and, similarly, the probability  of selecting a clique member with a reputation $r_k$ is

\begin{equation}
\label{s_c}
s_k^{(cl)}= \frac{r_kQ_k}{\sum_{i=0}^Lr_i (R_i +Q_i)}. 
\end{equation}

The main characteristic of the system's functionality is the probability $p_c$ of correct evaluation of a document. All three types of documents have to be taken into account: generic, `g', supporting the clique's views, `q', and contradicting the clique's views, `\=q':

\begin{equation}
\label{pc_avg_cl}
\begin{split}
p_c & = Prob(g,A)p_c(g,A)+Prob(g,F)(1-p_m(g,F))\\
& + Prob(q,A)p_c(q,A)+Prob(q,F)(1-p_m(c,F))\\
& + Prob(\bar{q},A)p_c(\bar{q},A)+Prob(\bar{q},F)(1-p_m(\bar{q},F)).
\end{split}
\end{equation}

Here, the probabilities of submitted  authentic, `A', or fake, `F', documents of types $g$, $q$, and $\bar{q}$ are

\begin{equation}
\label{prob_doc_cl} 
\begin{split}
Prob(g,A) & = (1-2p_{\lambda})\sum_{k=0}^LR_k a_k,\\
Prob(g,F) & = (1-2p_{\lambda}) \sum_{k=0}^LR_k (1-a_k),\\
Prob(q,A) & = p_{\lambda}\sum_{k=0}^LR_k a_k + \sum_{k=0}^LQ_k \gamma a_k,\\
Prob(q,F) & = p_{\lambda}\sum_{k=0}^LR_k (1-a_k) + \sum_{k=0}^LQ_k (1-\gamma a_k),\\
Prob(\bar{q},A) & = p_{\lambda}\sum_{k=0}^LR_ka_k \\
Prob(\bar{q},F) & = p_{\lambda}\sum_{k=0}^LR_k (1-a_k).
\end{split}
\end{equation}

In summary, the model of a community with a clique has three parameters, additional to the parameters of the model (\ref{caseL}) of a community without cliques: $f_{cl}$, the fraction of the clique in the community, $p_{\lambda}$, the proportion of documents matching the clique's agenda, and $\gamma$, the factor that stands for decreasing the probability of submitting authentic documents by clique members. 

System (\ref{caseL_1cl}), similar to (\ref{caseL}), has easy-to-guess equilibria at which $p_c=1$.

\begin{proposition}\label{propo_1cl_equil} 
		System (\ref{caseL_1cl}) has  equilibria at which $R_L=1-f_{cl}-R_0$, $R_0\in [0,1-f_{cl}]$, $Q_0=f_{cl}$, and all the other phase variables are equal to zero. At the equilibria $p_c=1$.
\end{proposition}

\begin{proof} Substitution of the equilibrium values of $R_k$ and $Q_k$ into (\ref{s_n}) and (\ref{s_c}) gives $s_L^{(reg)}=1$ while all the other $s_k^{(reg)}=0$ and $s_k^{(cl)}=0$. Then (\ref{pcind_cl}) yields $p_{c,ind}(\cdot,A)=1$ and $p_{m,ind}(\cdot,F)=0$. Next, (\ref{e_k_n}) and (\ref{e_k_c}) imply $e_k^{(reg)}=a_k$ and $e_k^{(cl)}=\gamma a_k$, meaning in particular, $e_L^{(reg)}=1$ and $e_0^{(cl)}=0$. These values and the putative equilibrium values of the phase variables make all the right-hand sides of (\ref{caseL}) equal to zero. Thus it is indeed an equilibrium. Similarly, it is directly verified that at the points of equilibria $p_c=1$.
\end{proof}

\begin{figure}[h]
	\centering
	\includegraphics[scale=0.6]{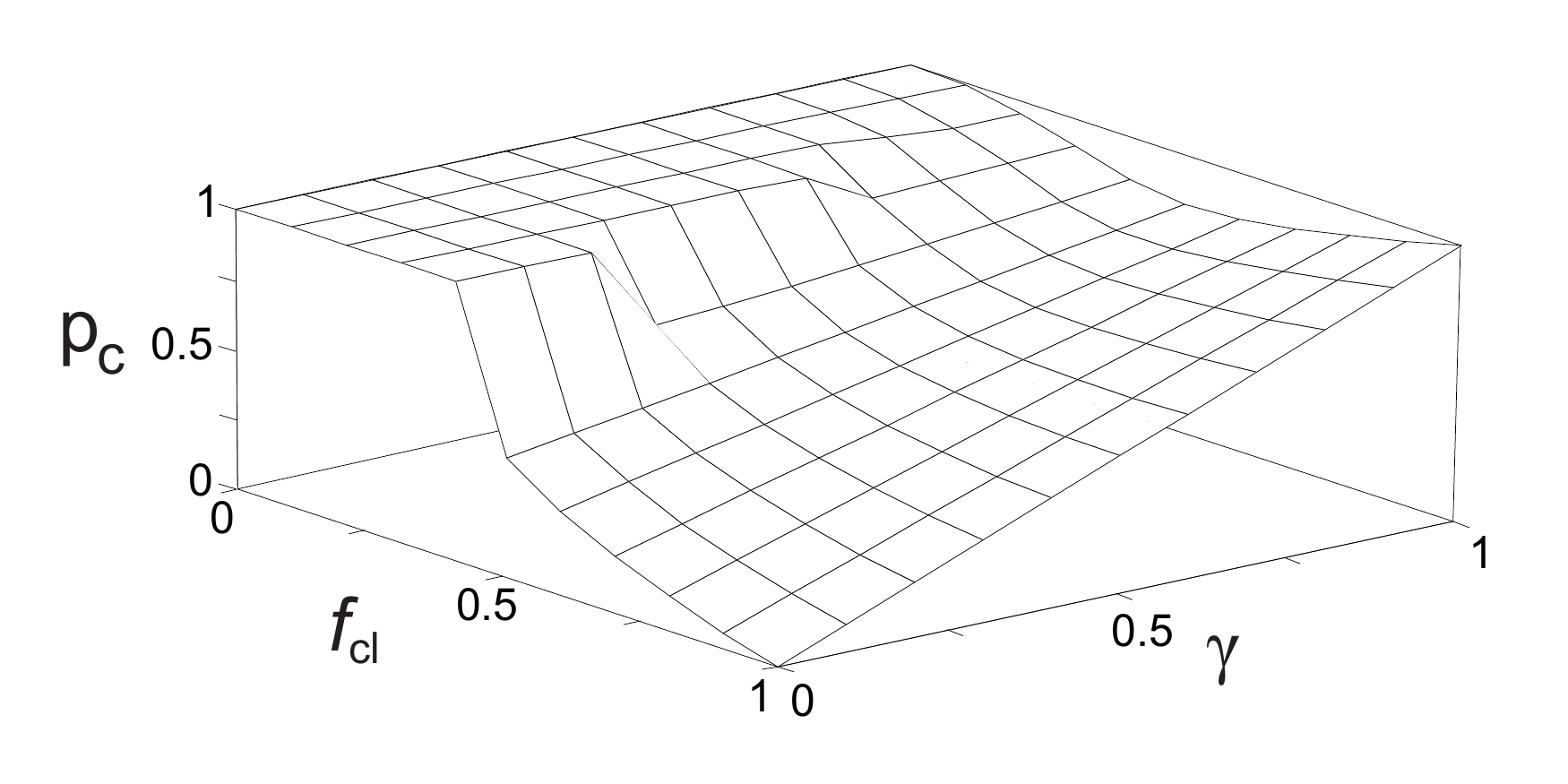}
	\caption{Probability of correct document evaluation $p_c$ in the model of a community with a clique as a function of the fraction of clique members in the community, $f_{cl}$, and the factor, decreasing the probability of submitting authentic documents by clique members, $\gamma$. At $t=0$, $R_6=1-f_{cl}$, $Q_6=f_{cl}$, while all the other $R_k=0$ and $Q_k=0$. Documents are evaluated correctly, $p_c=1$, only when the fraction of the clique $f_{cl}$ is small. $p_{\lambda}=0.01$, $\alpha=\sigma=0$.}
	\label{fig6}
\end{figure}

Figure \ref{fig6} shows the dependency of the probability of correct document evaluation $p_c$ on various $\gamma$ and $f_{cl}$ for $p_{\lambda}=0.01$, $\alpha=\sigma=0$, and $R_6(0)=1-f_{cl}$, $Q_6(0)=f_{cl}$, while all the other $R_k(0)$ and $Q_k(0)$ are equal to zero. The figure represents the system behavior at time $t=200$, when the system is close to equilibrium.

Consider two extreme cases shown in the figure, when there is no clique, $f_{cl}=0$, and when there are no regular members, $f_{cl}=1$. In the absence of a clique the equations for a community with a clique (\ref{caseL_1cl}) turn into the equations for a community without a clique (\ref{caseL}). The value of $\gamma$ becomes irrelevant, and $p_c=1$, as in Fig. \ref{fig_reps10_T100_Pc_v2.pdf}. In the absence of regular members the dynamics of the system is very simple. All documents, submitted by clique members, are evaluated by clique members only  and therefore are judged as authentic. The reputations of all the members only increase. Eventually, all the members have the maximal reputation, $r_L=1$. At this point, for all values of other parameters, $Prob(c,A)=\gamma$, $p_c(q,A)=1$, $p_m(q,F)=1$, and therefore $p_c=\gamma$ (Fig.\ref{fig6}). 

Simulations show that the system sets to the states with perfect evaluation, $p_c=1$, whenever the clique size $f_{cl}$ does not exceed 20\% of the whole community and $\gamma \leq 0.7$. When $\alpha$ and $\sigma$ both increase, $p_c=1$ for greater clique sizes (not shown).  

In the simulations above, the proportion $p_{\lambda}$ of documents, submitted by regular members and  related to the clique, was small, $p_{\lambda}=0.01$. The equilibria with $p_c=1$ from proposition \ref{propo_1cl_equil} should be independent of $p_{\lambda}$ according to the proposition's proof. Figure \ref{fig_pcavg_alpha_0_sigma_0_gamma_05_plambda.pdf} shows this is indeed the case for $\alpha=\sigma=0$, and $\gamma = 0.5$. It shows also that the states with the probability of correct document evaluation $p_c$ less than one depend on $p_{\lambda}$ only a little.

\begin{figure}[!th]
	\centering
	\includegraphics[scale=0.75]{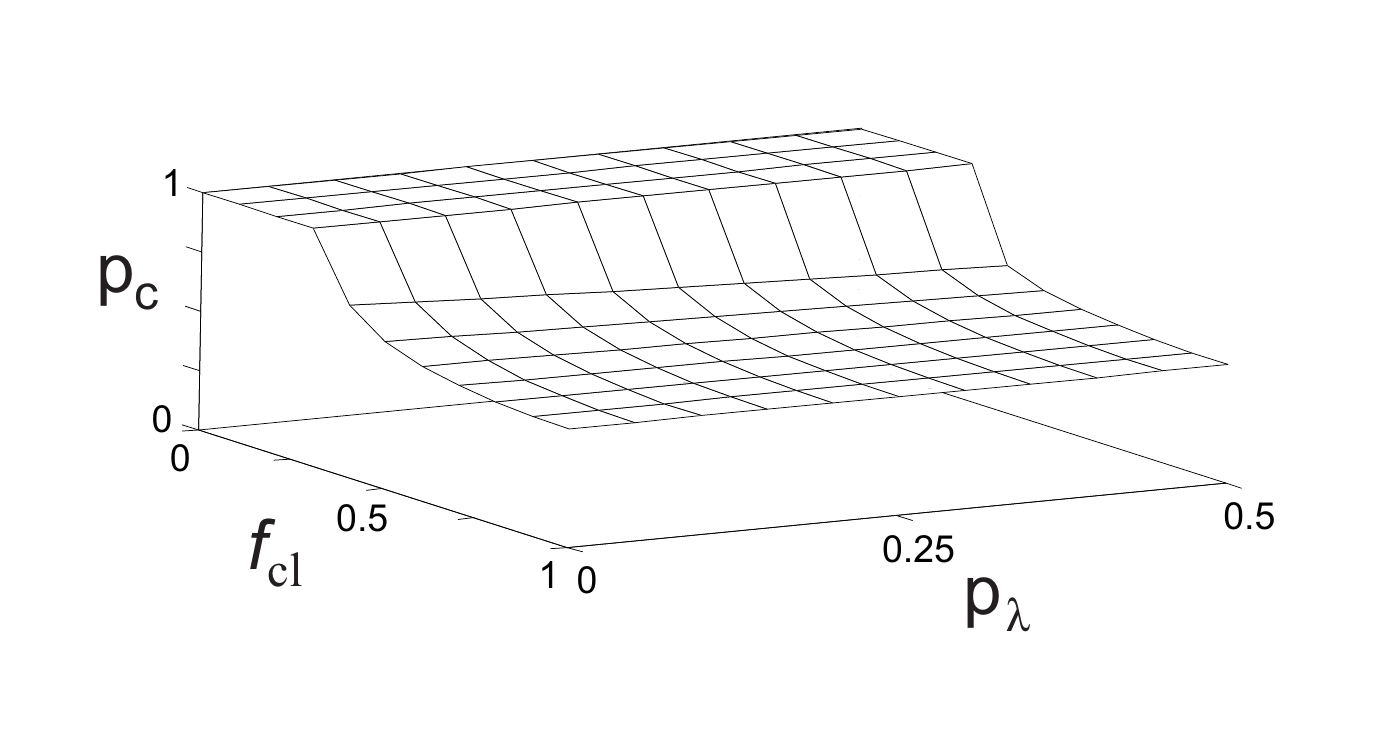}
	\caption{Probability of correct document evaluation $p_c$ in the equilibrium states weakly depends on the proportion $p_{\lambda}$ of documents, related to the clique's agenda; $p_c=1$ when the fraction of clique members in the community, $f_{cl}$, is less than 20\%. At $t=0$, $R_6=1-f_{cl}$, $Q_6=f_{cl}$, while all the other $R_k=0$ and $Q_k=0$. $\alpha=\sigma=0$, $\gamma = 0.5$.}
	\label{fig_pcavg_alpha_0_sigma_0_gamma_05_plambda.pdf}
\end{figure}

\CAAnewsection{6. Two antagonistic cliques}
\label{Two antagonistic cliques}
Cliques often form antagonistic pairs. In such a pair, the agendas of a clique and an `anti--clique' are diametrically opposite. As in the model with one clique, three types of documents are considered: generic, `g', supporting the clique's views, `q', and contradicting the clique's views, `\=q'. The antagonism of the cliques determines the probabilities of submission and correct evaluation of documents (Tables \ref{2cl_subm} and \ref{2cl_eval}). All the parameters in the tables play the same roles as in the model with one clique.

Table \ref{2cl_subm} defines the probabilities of submitting authentic and fake documents by regular,  clique, and anti-clique members. For example, according to the first line, regular members with a reputation $r_k$ submit authentic generic documents with the probability $(1-2p_{\lambda}) a_k$ while members of the cliques never submit such documents. The second line shows that the probability that regular members with a reputation $r_k$ submit fake generic documents is equal to $(1-2p_{\lambda})(1-a_k)$. Cliques members do not submit fake generic documents, and so on.

Table \ref{2cl_eval} defines the probabilities of correct evaluation of particular documents by regular,  clique, and anti-clique members. The probabilities $c_k$ of correct evaluation of documents by members with reputations $r_k$ are defined by (\ref{a_c_k}). As the table shows, it is assumed that the members of both cliques, when selected to be evaluators, don't care about generic documents that are not related to their agenda; they evaluate those documents by flipping a coin. All the documents that support the position of one clique the members of the antagonistic clique  evaluate as fake. All the documents that contradict the position of one clique the members of the antagonistic clique  evaluate as authentic.

Let $U_k$, $k=0,\cdots,L$, be the proportions of the members of the anti-clique with reputations $r_k$. The equations for $U_k$ have the same form as for  regular members and  members of the clique (see \ref{caseL}) since the rules of the community are the same for all the members.

\begin{table}
	\vspace{\baselineskip}
	\caption{Probabilities of {\emph{submitting}} various types of documents by regular, clique, and anti-clique members with the reputations $r_k$, $k=0,\cdots,L$. }  
	\vspace{\baselineskip}
	\begin{center}
		\begin{tabular}{l *{3}{r}}
			\toprule
			Document type & Regular members & Clique members & Anti-clique members\\
			\bottomrule
			Generic, authentic & $(1-2p_{\lambda}) a_k$ & $0$ & $0$ \\
			\midrule
			Generic, fake & $(1-2p_{\lambda})(1-a_k)$ &  $0$ &  $0$\\
			\midrule
			Clique, authentic & $p_{\lambda}a_k$ & $\gamma a_k$ & $0$\\
			\midrule
			Clique, fake & $p_{\lambda}(1-a_k)$ & $1-\gamma a_k$ & $0$\\
			\midrule
			Anti-clique, authentic & $p_{\lambda}a_k$ & $0$ & $\gamma a_k$\\
			\midrule
			Anti-clique, fake & $p_{\lambda}(1-a_k)$ & $0$  & $1-\gamma a_k$\\
			\bottomrule
		\end{tabular}
	\end{center}
	\label{2cl_subm}
\end{table}
\vspace{\baselineskip}

\begin{table}
	\vspace{\baselineskip}
	\caption{Probabilities of correct {\emph{evaluation}} of various types of documents by regular, clique, and anti-clique members with reputation $r_k$, $k=0,\cdots,L$. }  
	\vspace{\baselineskip}
	\begin{center}
		\begin{tabular}{l *{3}{r}}
			\toprule
			Document type & Regular members & Clique members & Anti-clique members\\
			\bottomrule
			Generic, authentic & $c_k$ & $0.5$ & $0.5$\\
			\midrule
			Generic, fake & $c_k$ &  $0.5$  & $0.5$\\
			\midrule
			Clique, authentic & $c_k$ & $1$ & $0$ \\
			\midrule
			Clique, fake & $c_k$& $0$  & $1$ \\
			\midrule
			Anti-clique, authentic & $c_k$ & $0$ & $1$\\
			\midrule
			Anti-clique, fake & $c_k$ & $1$ & $0$ \\
			\bottomrule
		\end{tabular}
	\end{center}
	\label{2cl_eval}
\end{table}

\begin{equation}
\label{caseL_2cl}
\begin{split}
\frac{dR_{0}}{dt} & = -R_{0}e_0^{(reg)}+R_{1}(1-e_1^{(reg)})\\
& \vdots \\
\frac{dR_{k}}{dt} & = R_{k-1}e_{k-1}^{(reg)}-R_{k}+R_{k+1}(1-e_{k+1}^{(reg)}), \qquad k =1,\cdots,L-1,\\
& \vdots \\
\frac{dR_{L}}{dt} & = R_{L-1}e_{L-1}^{(reg)}-R_{L}(1-e_{L}^{(reg)})\\
\vspace{0.1in}
\frac{dQ_{0}}{dt} & = -Q_{0}e_0^{(cl)}+Q_{1}(1-e_1^{(cl)})\\
& \vdots \\
\frac{dQ_{k}}{dt} & = Q_{k-1}e_{k-1}^{(cl)}-Q_{k}+Q_{k+1}(1-e_{k+1}^{(cl)}), \qquad k =1,\cdots,L-1,\\
& \vdots \\
\frac{dQ_{L}}{dt} & = Q_{L-1}e_{L-1}-Q_{L}^{(1)}(1-e_{L}^{(cl)})\\
\vspace{0.1in}
\frac{dU_{0}}{dt} & = -U_{0}e_0^{(a\textrm{-}cl)}+U_{1}(1-e_1^{(a\textrm{-}cl)})\\
& \vdots \\
\frac{dU_{k}}{dt} & = U_{k-1}e_{k-1}^{(a\textrm{-}cl)}-U_{k}+U_{k+1}(1-e_{k+1}^{(a\textrm{-}cl}), \qquad k =1,\cdots,L-1,\\
& \vdots \\
\frac{dU_{L}}{dt} & = U_{L-1}e_{L-1}^{(a\textrm{-}cl)}-U_{L}(1-e_{L}^{(a\textrm{-}cl)}).
\end{split}
\end{equation} 

In (\ref{caseL_2cl}), $e_k^{(reg)}$ and $e_k^{(cl)}$ have the same meaning as in (\ref{e_k_n}-\ref{e_k_c}). Similarly, $e_k^{(a\textrm{-}cl)}$, $0\leq k \leq L$, stand for the probability that the documents, submitted by anti-clique members with a reputation $r_k$, are evaluated as authentic. It is readily verified that the sum of all $U_k$ does not change over time. We set the sum to be equal to $f_{a\textrm{-}cl}$. The sum of all $Q_k$ does not change with time as well. This sum is set to $f_{cl}$. Finally, the sum of all $R_k$ is always equal to $1-f_{cl}-f_{a\textrm{-}cl}$.

System (\ref{caseL_2cl}) has equilibria similar to those of system (\ref{caseL_1cl}).

\begin{proposition} 
	\label{propo_2cl_equil}
	System (\ref{caseL}) has equilibria at which $R_L=1-f_{cl}-f_{a\textrm{-}cl}-R_0$, $R_0 \in [0, 1-f_{cl}-f_{a\textrm{-}cl}]$, $Q_0=f_{cl}$, $U_0=f_{a\textrm{-}cl}$, and all the other phase variables are equal to zero. At the equilibria $p_c=1$. 
\end{proposition}

In accord with Table \ref{2cl_subm}, and similar to the definition of $e_k^{(cl)}$, the probability that an (anti-clique) document, submitted by a member of the anti-clique is evaluated as authentic, is given by the formula 

\begin{equation}
\label{e_k_q}
e_k^{(a\textrm{-}cl)} = \gamma a_k p_c(\bar{q},A) + (1-\gamma a_k) p_m(\bar{q},F),  \qquad k=0,\cdots,L.
\end{equation}

The probabilities of correct evaluation, $p_c$, and false evaluation, $p_m$,  in (\ref{e_k_n}), (\ref{e_k_c}), and (\ref{e_k_q})  depend on the corresponding probabilities for individual evaluators ( cf.(\ref{p_c})). For example,

$$ p_c(g,A)= p_{c,ind}^2(g,A)(3-p_{c,ind}(g,A)). $$

In accord with Table \ref{2cl_eval}, 

\begin{equation}
\label{pcind_2cl} 
\begin{split}
p_{c,ind}(g,A) & = \sum_{k=0}^Ls_k^{(reg)} c_k+0.5\sum_{k=0}^L(s_k^{(cl)}+s_k^{(a\textrm{-}cl)}), \\
p_{m,ind}(g,F) & =\sum_{k=0}^Ls_k^{(reg)} (1-c_k)+0.5\sum_{k=0}^L(s_k^{(cl)}+s_k^{(a\textrm{-}cl)}), \\
p_{c,ind}(q,A) & = \sum_{k=0}^Ls_k^{(reg)} c_k+ \sum_{k=0}^Ls_k^{(cl)}, \\
p_{m,ind} (q,F) & = \sum_{k=0}^L s_k^{(reg)} (1-c_k)+\sum_{k=0}^Ls_k^{(cl)}, \\
p_{c,ind}(\bar{q},A) & = \sum_{k=0}^Ls_k^{(reg)} c_k+\sum_{k=0}^Ls_k^{(a\textrm{-}cl)}, \\
p_{m,ind}(\bar{q},F) & = \sum_{k=0}^Ls_k^{(reg)} (1-c_k)+\sum_{k=0}^Ls_k^{(a\textrm{-}cl)}.
\end{split}
\end{equation}

\noindent
Here, $c_k$ are defined as in (\ref{a_c_k}). 

The probability  of selecting a regular member with a reputation $r_k$ is  

\begin{equation}
\label{s_n_2cl}
s_k^{(reg)}= \frac{r_kR_k}{\sum_{i=0}^Lr_i (R_i+Q_{i}+U_{i})},
\end{equation}

\noindent
and, similarly, the probability  of selecting a clique member with a reputation $r_k$ is

\begin{equation}
\label{s_c_2cl}
s_k^{(cl)}= \frac{r_kQ_k}{\sum_{i=0}^Lr_i (R_i+Q_{i}+U_{i})}, 
\end{equation}

\noindent
and the probability  of selecting an anti-clique member with a reputation $r_k$
\begin{equation}
\label{s_q_2cl}
s_k^{(a\textrm{-}cl)}= \frac{r_kU_k}{\sum_{i=0}^Lr_i (R_i+Q_{i}+U_{i})}. 
\end{equation}

The formula for $p_c$ is the same as in the case of one clique (\ref{pc_avg_cl}). However, some of the terms change:  

\begin{equation}
\label{prob_doc_2cl}
\begin{split}
Prob(g,A) & = (1-2p_{\lambda})\sum_{k=0}^LR_k a_k,\\
Prob(g,F) & = (1-2p_{\lambda}) \sum_{k=0}^LR_k (1-a_k),\\
Prob(q,A) & = p_{\lambda}\sum_{k=0}^LR_k a_k + \sum_{k=0}^LQ_k \gamma a_k,\\
Prob(q,F) & = p_{\lambda}\sum_{k=0}^LR_k (1-a_k) + \sum_{k=0}^LQ_k (1-\gamma a_k),\\
Prob(\bar{q},A) & = p_{\lambda}\sum_{k=0}^LR_ka_k + \sum_{k=0}^LU_k \gamma a_k, \\
Prob(\bar{q},F) & = p_{\lambda}\sum_{k=0}^LR_k (1-a_k)+ \sum_{k=0}^LU_k (1-\gamma a_k).
\end{split}
\end{equation}

In summary, the system with two antagonistic cliques has the following parameters: the relative size of the first clique, $f_{cl}$,  the relative size of the antagonistic clique, $f_{a\textrm{-}cl}$,  the proportion of documents matching the clique's agenda, $p_{\lambda}$, and  the factor $\gamma$, $0 \leq \gamma \leq 1$, that decreases the probabilities of submitting authentic documents by clique and anti-clique members. For simplicity in what follows it is assumed that both cliques have the same size: $f_{cl}=f_{a\textrm{-}cl}$, $f_{cl}\leq 0.5$. When $f_{cl}=0.5$ there are no regular members in the community, $1-2f_{cl}=0$.

The first experiment was to explore whether the presence of an anti-clique improves the quality of document evaluation. The experiment shows that the improvement does take place but is not large. 
Figure \ref{fig_2cl_pc1_0_01_alpha_0_sigma_0.pdf} shows a typical behavior of the probability of correct document evaluation $p_c$ for various values of $f_{cl}$ and $\gamma$ at $t=200$ when the system is close to equilibrium. In the case, shown in the figure, $L=10$, $\alpha=\sigma=0$, at $t=0$, $R_6=0.4$, $Q_6=0.3$, and $U_6=0.3$ and all the other $R_k$, $Q_k$, $U_k$ were equal to zero. Comparison with the behavior of the community with one clique (Fig.\ref{fig6}) shows that the region of the parameters $f_{cl}$ and $\gamma$, for which $p_c=1$, in the case of two cliques is greater.

\begin{figure}[!th]
	\centering
	\includegraphics[scale=0.6]{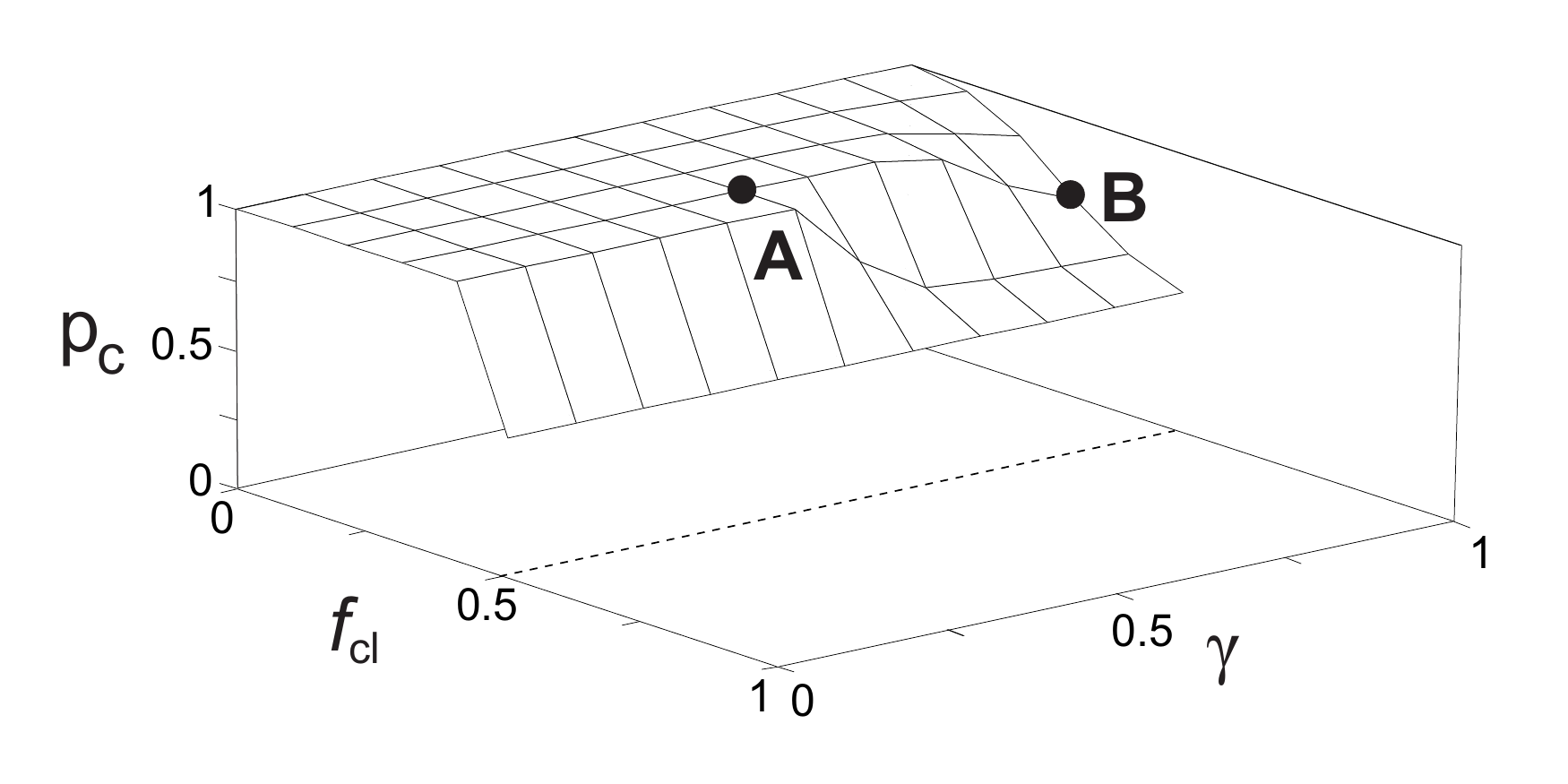} 
	\caption{Probability of correct document evaluation $p_c$ in a system with two cliques for various values of the fraction of clique members in the community, $f_{cl}$, and the factor, decreasing the probability of submitting authentic documents by clique members, $\gamma$. At $t=0$, $R_6=0.4$, $Q_6=0.3$, and $U_6=0.3$, while all the other components $R_k$, $Q_k$, and $U_k$ are equal to zero. $\alpha=\sigma=0$.  Details of the model behavior at points A and B are described in the text.}
	\label{fig_2cl_pc1_0_01_alpha_0_sigma_0.pdf}
\end{figure}

In accord with Proposition \ref{propo_2cl_equil}, in the states with $p_c=1$, such as at point A of  Fig.\ref{fig_2cl_pc1_0_01_alpha_0_sigma_0.pdf}, we have $R_{10}=1-2f_{cl}$, $Q_0=f_{cl}$, $U_0=f_{cl}$.  The system pushes regular agents into just two categories, with zero and 100\% reputation. All clique members end up with zero reputation.

When $f_{cl}=0.3$ and $\gamma = 1$ (Fig.\ref{fig_2cl_pc1_0_01_alpha_0_sigma_0.pdf}, point B) the evolution of the system is different. In the corresponding steady state there are clique members with non-zero and even 100\% reputation. These members can be selected to serve as evaluators. Because of that there are many mistakes in document evaluation, $p_c \approx 0.7$.

\CAAnewsection{7. Discussion}
\label{Related work}
Most of the related work is done in the context of academic peer-reviewing and e-commerce.

Academic peer-review models are often considerably more complex compared to the presented model which makes direct comparison with our model challenging. They include academic-specific factors such as the resources available to researchers \cite{CGS14}, reputation of a journal \cite{KT16}, impact of rational referees, who might not have incentives to see high quality publications other than their own \cite{TH11}, and so on. Nevertheless, these studies have ideas that could be applied to the model considered. In particular in \cite{H12}, documents are evaluated by highly qualified `experts', who make few mistakes, and less qualified `readers', who make more mistakes. It has been shown that the evaluations by large numbers of readers (up to 100) have better accuracy compared to the evaluations obtained by several experts. The presented model can be generalized to include not three but any number of evaluators to exploit this mechanism.   

Peer-to-peer interactions and associated reputation and trust models is a field that has a lot of interest because of online commerce (e.g. eBay and Amazon) \cite{XL03,KSH03,MH02,HBC14,RF00,ZM00}, sharing economy (e.g. Uber, Lyft, AirBnB, Homeaway) \cite{Q17}, blockchain networks \cite{TL18}, and sharing resources and information \cite{CB13,CH16}. The computational models of reputation and trust vary in how individuals interact and use the results of interactions. Most of this work is based on agent-based simulations. Theoretical results are rare. Our approach is most closely related to EigenTrust reputation management system in peer-to-peer networks \cite{KSH03}. In EigenTrust, peers accumulate global trust values from peer-to-peer interactions. Greater global trust values increase the chances of peers to evaluate others. In our model, not individuals but subsets of members who have the same value of reputation are considered. It makes the model more tractable analytically and computationally. 

Mathematically, our model belongs to the class of compartmental models (see review \cite{BG12}). Such models have been used in numerous contexts, from epidemiology \cite{H00} to the transport of pollutants in ecosystems\cite{G91}. The system in \cite{BS05} is particularly related to the present study. It describes interactions between altruists and defectors. As in the model considered here, the model in \cite{BS05} has a continuum of stable equilibrium states that are not asymptotically stable.

\CAAnewsection{8. Conclusions and future work}
\label{Conclusions and future work}
 
We describe a model of a network community, in which fake and authentic  documents, submitted by the community members, can be effectively evaluated by peers. The mechanism of evaluation is based on the intuition that judgment should be delegated to skilled and reputable evaluators. Similar mechanisms have been considered earlier \cite{KSH03}, but as far as we know not in population models. The population model approach allowed us to explore possible dynamics of the community assuming large number of its members. In traditional agent-based simulations the size of model communities is limited.    

The proposed model is flexible and can be developed in multiple ways.  Some parameters should be refined on the basis of sociological studies. In particular here we set the probability of submitting authentic documents, $a(r)$, and the probability of correct evaluation of documents, $c(r)$, to be generic quadratic functions of the agent's reputation $r$. Further examination should determine for what values of the parameters $\alpha$ and $\sigma$, functions $a(r)$ and $c(r)$ approximate data in the best way. Second, it could be that some clique members are not totally preoccupied with the clique agenda, as assumed in the present model, but also submit and correctly evaluate generic documents. To what extent would this impact the model?

Other parameters of the model can be chosen by developers of the system.  The presented results show that within a certain range of such parameters the model has the desired behavior. Future theoretical analyses, including global sensitivity analysis \cite{MK08}, may improve such choices. Analyses may show that performance of the model gets better if more than three  evaluators evaluate every document, reveal the optimal number and distribution of reputation levels (potentially considering continuous distribution of reputations and partial differential equations), determine that the system functions better if earning reputation is harder than losing it, etc.

In the present study we model a closed system; similar systems are often considered in chemical kinetics. Adding new community members with certain rate will transform the system into an open system. It would be interesting to explore the changes of the mathematical properties of the system under this transformation.

Can a community modeled in the present study function in the real world?
It depends on the extent to which the model assumptions are correct.
Two assumptions can be met relatively easy. 
One is that all the documents are verifiable. 
For example, if the description says that the photo depicts an event happening at some place at some time then the image should have the corresponding EXIF-encoded GPS coordinates and time. GPS-connected smartphones and cameras get this information automatically. In a real-world community, an authentic document is also expected to be novel and not a copy or a slightly changed version of an already posted document. Current technologies allow to meet these requirements as well. 

The second assumption is that the community members, selected to serve as evaluators, should really do it and do it quickly. This is a matter of the community discipline. Many journals demonstrate such a discipline is possible to maintain.

\CAAnewsection{ACKNOWLEDGMENTS}
The author is grateful to L.~Ritter and anonymous reviewer for valuable comments. The idea of a network community with a decentralized evaluation mechanism to filter fake information is joint with E.~Belyakov. It is part of our  project of an Internet resource for sharing, collecting and organizing event-related media information.

\end{document}